\documentclass[11pt]{article}
\usepackage{amsmath}
\usepackage{amsfonts}
\usepackage{graphicx}
\usepackage{latexsym}
\usepackage{url}
\usepackage{color} 
\usepackage{dsfont}
\usepackage{amssymb}
\usepackage{amsthm}
\usepackage{comment}
\usepackage{enumerate}%\begin{enumerate}[{(}i{)}]
\excludecomment{en-text}
\includecomment{jp-text}
\excludecomment{comment}

\newcommand{\R}{\mathbb{R}}

\newcommand{\E}{\mathbb{E}}
\newcommand{\PP}{\mathbb{P}}

\newcommand{\toop}{\stackrel{\PP}{\longrightarrow}}
\newcommand{\ucp}{\stackrel{u.c.p.}{\longrightarrow}}

\newcommand{\eqschw}{\stackrel{d}{=}}

\newcommand{\bee}{\begin{equation}}
\newcommand{\eee}{\end{equation}}
\newcommand{\bea}{\begin{eqnarray}}
\newcommand{\eea}{\end{eqnarray}}
\newcommand{\bean}{\begin{eqnarray*}}
\newcommand{\eean}{\end{eqnarray*}}

\renewcommand{\theequation}{\arabic{section}.\arabic{equation}}

\newtheorem{prop}{Proposition}[section]
\newtheorem{cor}[prop]{Corollary}

\newtheorem{theo}[prop]{Theorem}

\begin{document}

\title{Asymptotic behavior of local times related statistics for fractional Brownian motion}
\author{Mark Podolskij  \thanks{Department of Mathematics, Aarhus University,  Denmark, mpodolskij@math.au.dk.} \and 
Mathieu Rosenbaum \thanks{CMAP, \'Ecole Polytechnique, France, mathieu.rosenbaum@polytechnique.edu.} 
}

\maketitle

\begin{abstract}

\noindent We consider high frequency observations from a fractional Brownian motion. Inspired by the work of Jean Jacod in a diffusion setting, we investigate the asymptotic
behavior of various classical statistics related to the local times of the process. We show that as in the diffusion case, these statistics indeed converge to some local times up to a constant factor. As a corollary, we provide limit theorems for the quadratic variation of the absolute value of a fractional Brownian motion.

\ \

\noindent {\it Keywords}: Fractional Brownian motion, local times, quadratic variation, functional limit theorems \
\bigskip

%{\it AMS 2010 subject classifications.} 

\end{abstract}

\section{Introduction} \label{sec1}
\setcounter{equation}{0}
\renewcommand{\theequation}{\thesection.\arabic{equation}}

Local times are fundamental objects in the theory of stochastic processes. The local time of the process $X$ at time $t$ and level $x$ essentially measures the time spent by the process at point $x$ between time $0$ and time $t$. More precisely, let $X=(X_t)_{t \geq 0}$ be defined on a filtered probability space $(\Omega, \mathcal{F}, 
(\mathcal{F}_t)_{t \geq 0}, \mathbb P)$. Recall that the occupation measure $\nu$ associated with $X$ is defined via
\begin{align} \label{occupation}
\nu_t(A) := \int_0^t 1_{A} (X_s) ds,
\end{align}
for any Borel measurable set $A \subset \R$.  If the measure $\nu_t$ admits a Lebesgue density, we call it the local time of $X$ and denote it by 
\begin{align} \label{ltime}
L_t (x) := \frac{d \nu_t}{dx} (x).
\end{align}
An immediate consequence of the existence of local times is the occupation times formula, which states the identity 
\begin{align} \label{otf}
\int_0^t f(X_s) ds = \int_{\R} f(x) L_t (x) dx
\end{align}
with $f$ being a Borel function on $\R$.\\

The theory of occupation measures and local times has been intensively studied in the literature for several decades. A general criterion for existence of local times is provided in \cite{GH}. Properties of the local time of a Brownian motion, or more generally of a continuous diffusion process, have been particularly investigated. We refer to the monograph \cite{RY} for a comprehensive study.  In \cite{B} the author focuses on sample paths properties in the case of stationary Gaussian processes.\\ 
 
Then a natural question is that of the estimation of local times based on discrete high frequency observations of the process $X$ over $[0,t]$: $X_0, X_{1/n}, X_{2/n}, \ldots, X_{[nt]/n}$, where the asymptotic setting is that $n$ goes to infinity. From \eqref{ltime}, a natural estimator of $L_t (x)$ is given by the statistic
\[
\frac{u_n}{n}\sum_{i=1}^{[nt]} 1_{\left\{|X_{i/n}-x| \leq 1/u_n \right\}},
\]
where $u_n$ is a sequence of positive real numbers with $u_n \to \infty$. In general we may consider a larger class of statistics, which have the form 
\begin{align} \label{Vg}
V(g)_t^n:= \frac{u_n}{n}\sum_{i=1}^{[nt]} g\left(u_n (X_{i/n}-x)\right) \qquad \text{with} \qquad g \in L^1(\R). 
\end{align}  
Indeed, the statistic $V(g)_t^n$ is, up to a constant, a good estimator of $L_t(x)$. This can be seen from the following intuitive argument, which follows from the occupation time formula:
\begin{align*}
V(g)_t^n &\approx u_n \int_0^t g(u_n (X_s-x)) ds = u_n \int_{\R} g(u_n (y-x)) L_t (y) dy \\[1.5 ex]
&=  \int_{\R} g(z) L_t (x+z/u_n) dz \toop L_t(x)  \int_{\R} g(z) dz,
\end{align*} 
where the last convergence follows from the continuity of the local time in space.  Despite this simple intuition, the first approximation in the above argument is not trivial to prove (in particular it is wrong if 
$u_n /n \to \infty$). Such convergence results have been shown in the framework of Brownian motion in 
\cite{B86, B89} and they were extended in \cite{J98} to the setting of continuous diffusion processes (in the latter article the author has also shown the associated central limit theorems). In \cite[Theorem 4]{J04} the consistency of $V(g)_t^n$ has been proved for linear fractional stable motions.\\   

In \cite{J98}, with the purpose of statistical applications, a more general class of functionals than $V(g)_t^n$ is actually considered in the setting of continuous diffusion processes. More specifically, for a measurable function $f: \R^2 \to \R$ satisfying certain conditions to be specified below, the author studies statistics of the form
\begin{align} \label{Vf}
V(f)_t^n:= \frac{u_n}{n}\sum_{i=0}^{[nt]} f \left(u_n (X_{i/n}-x), u_n \Delta_i^n X \right) \qquad \text{where} \qquad 
\Delta_i^n X = X_{(i+1)/n} - X_{i/n}.
\end{align} 

In this work, we wish to somehow extend the approaches in \cite{J98} and \cite{J04}. Indeed we want to investigate the behavior of $V(f)_t^n$ in the case where the underlying process $X$ satisfies 
\begin{align} \label{fbm}
X_t = \sigma B_t^H,
\end{align} 
where $\sigma>0$ is a scale parameter and $(B_t^H)_{t \geq 0}$ is a standard fractional Brownian motion with Hurst parameter $H \in (0,1)$, that is a zero mean Gaussian process with covariance kernel given by
\[
\E[B_t^H B_s^H]= \frac{1}{2} \left(t^{2H} + s^{2H} -|t-s|^{2H} \right).
\] 
The fractional Brownian motion has stationary increments and it is self-similar with parameter $H$: $(B_{at}^H)_{t \geq 0} \eqschw (a^HB_{t}^H)_{t \geq 0}$ for any $a>0$. Furthermore, local times of $B^H$ exist for any $H \in (0,1)$ and are known to be continuous in time and space, see \cite{B2}.\\

Studying high frequency statistics in the context of fractional Brownian motion is very natural. Indeed, this process is probably the simplest non-trivial continuous process outside the semi-martingale world. Furthermore, beyond its obvious theoretical interest, the fractional Brownian motion is widely used in various applications, particularly in finance, see among others \cite{CR,CS,GJR}. Finally, note that for the specification of the function $f$ in \eqref{Vf}, there will be two cases of particular interest (both refer to $u_n=n^H$):
\begin{align} \label{example}
f_1(y,z)= 1_{\{y(y+z)<0\}} \qquad \text{and} \qquad f_2(y,z)= -2y(y+z) 1_{\{y(y+z)<0\}}.
\end{align}
We have that $V(f_1)_t^n$ counts the (scaled) number of crossings of level $x$ and the statistics  
$V(f_2)_t^n$ will be useful for some applications.\\

We give in the next section our assumptions and our main result on the behavior of $V(f)_t^n$. As a corollary, we provide a panorama of the different  limit theorems obtained for the quadratic variation of the absolute value of a fractional Brownian motion, depending on the Hurst parameter. The proofs are gathered in Section \ref{sec3}. 

\section{The setting and main results} \label{sec2}
\setcounter{equation}{0}
\renewcommand{\theequation}{\thesection.\arabic{equation}}

In order to formulate our main result we will require some conditions on the function $f$: 
\newline \newline
(A-$\gamma$) We assume that $|f(y,z)|\leq h_1(y) h_2(z)$ where $h_1,h_2$ are non-negative continuous functions such that $h_1$ tends to zero at infinity,
\begin{align} \label{assumption}
\int_{\R} |y|^{\gamma} h_1(y) dy < \infty
\end{align}
and $h_2$ has polynomial growth. 
\newline \newline
We notice that if condition (A-$\gamma$) is satisfied for some function $h_1$, then it also holds for $h_1^p$ for any $p \geq 1$. Furthermore, (A-$\gamma$) implies  (A-$\gamma'$) for any $0\leq \gamma' \leq \gamma$. The assumption (A-$\gamma$) might be relaxed (cf. \cite[Hypothesis (B-$\gamma$)]{J98}), but it suffices for applications we have in mind.

Our main theorem is the following statement.

\begin{theo} \label{th1}
Assume Condition (A-$2$) is satisfied. 
If $u_n \to \infty$ and $u_n/n \to 0$ we obtain the uniform convergence in probability
\begin{align} \label{ucp}
V(f)_t^n \ucp V(f)_t := L_t(x) \int_{\R^2} f(y,z) \phi_{\sigma^2} (z) dydz,
\end{align}  
where $\phi_{\sigma^2}$ denotes the density of $\mathcal N(0, \sigma^2)$. 
\end{theo}
We remark that the integral on the right hand side of \eqref{ucp} is indeed finite due to assumption (A-$2$).
We apply the result of Theorem \ref{th1} to the functions $f_1, f_2$ defined at \eqref{example}. First of all, we notice that they both satisfy assumption (A-$2$). Indeed,  the condition $x(x+y)<0$ implies that
$|x|<|y|$ and we obtain that 
\[
|f_1(x,y)| \leq (\max\{1, |x|^{p}\})^{-1} \max\{1, |y|^{p}\}, \qquad 
|f_2(x,y)| \leq 4 |x|(\max\{1, |x|^{p}\})^{-1} |y|\max\{1, |y|^{p}\}
\] 
for any $p>0$. Choosing $p>4$  ensures the validity of \eqref{assumption} for $\gamma=2$. We obtain
the following corollary.

\begin{cor} \label{cor1}
Define $u_n=n^H$ and let $N\sim \mathcal{N}(0,\sigma^2)$. Then we obtain
\begin{align*} 
V(f_1)_t^n \ucp \E[|N|]L_t(x) \qquad \text{and} \qquad V(f_2)_t^n \ucp \frac{1}{3}\E[|N|^3]L_t(x). 
\end{align*}  
\end{cor} 

Our next application of Theorem \ref{th1} is the weak limit theorem for the quadratic variation of the 
process $|X|$. 

\begin{prop} \label{prop2}
We define $\rho_k = \text{cov}(B_1^H, B_{k+1}^H - B_{k}^H)$ and 
$v^2= 2\sigma^4(1+ 2\sum_{k=1}^{\infty} \rho_k^2)$. Let $W=(W_t)_{t \geq 0}$ denote the standard Brownian motion. \\
(i) If $H <1/2$ we obtain the uniform convergence in probability
\[
n^{H-1} \sum_{i=0}^{[nt]} \left(  (n^H \Delta_i^n |X|)^2 -\sigma^2 \right) \ucp \frac{1}{3}\E[|N|^3]L_t(0).
\]
(ii) If $H=1/2$ we obtain the functional convergence for the Skorokhod topology
\[
\frac{1}{\sqrt{n}} \sum_{i=0}^{[nt]} \left(  (n^H \Delta_i^n |X|)^2 -\sigma^2 \right) \Rightarrow 
vW_t + \frac{1}{3}\E[|N|^3]L_t(0),
\] where $W$ is independent of $L$.

(iii) If $H \in (1/2, 3/4)$ we obtain the functional convergence for the Skorokhod topology
\[
\frac{1}{\sqrt{n}} \sum_{i=0}^{[nt]} \left(  (n^H \Delta_i^n |X|)^2 -\sigma^2 \right) \Rightarrow 
vW_t .
\]
(iv) If $H\in (3/4,1)$ we obtain the functional convergence for the Skorokhod topology
\[
n^{1-2H} \sum_{i=0}^{[nt]} \left(  (n^H \Delta_i^n |X|)^2 -\sigma^2 \right)  \Rightarrow \sigma^2 R_t,
\]
where $R$ is the Rosenblatt process.
\end{prop}
Hence, interestingly, we see that the asymptotic behavior of the quadratic variation of the absolute value of a fractional Brownian motion differs from that
of a fractional Brownian motion when $H\leq 1/2$. In this case, the local time at zero appears in the limit.\\

Note that in \cite{R}, the author considers max-stable processes whose spectral processes are exponential martingales associated to a Brownian motion. He introduces estimators of the integral of the extreme value index function and establishes asymptotic properties of these estimators thanks to results in \cite{J98}. Proposition \ref{prop2} could  then used to extend this framework to exponential martingales associated to a fractional Brownian motion, leading to a more general and useful class of max-stable processes.

\section{Proofs of Theorem \ref{th1} and Proposition \ref{prop2}} \label{sec3}
\setcounter{equation}{0}
\renewcommand{\theequation}{\thesection.\arabic{equation}}
Throughout this section all positive constants are denoted by $C$ (or $c$), or by $C_p$ if they depend on the external parameter $p$, although they may change from line to line. For the sake of exposition we only consider the local time $L_t(x)$ at $x=0$ and set $\sigma=1$. 

\subsection{Introduction for the proof of Theorem \ref{th1}}

Before we proceed with formal proofs let us give important ideas for the strategy of proof.
First we note that it suffices to prove the pointwise result $V(f)_t^n \toop V(f)_t$ for any fixed $t$, since the statistic $V(f)_t^n$ is increasing in $t$ and the limit $V(f)_t$ is continuous in $t$. Thus, without loss of generality we set $t=1$. Next, we introduce the statistic
\begin{align*}
\overline{V}(f)^n:= \frac{u_n}{n}\sum_{i=0}^{n} F(u_n X_{i/n}) \qquad \text{with} \qquad
F(x):=\E[f(x, \mathcal N(0, 1))].
\end{align*}
It follows from \cite[Theorem 4]{J04} that the convergence 
\[
\overline{V}(f)^n \toop L_1(0) \int_{\R} F(x) dx
\]
holds as $n \to \infty$. Hence, we are left to proving 
\begin{align}  \label{convergence1}
V(f)^n_1 - \overline{V}(f)^n \toop 0,
\end{align} 
to show Theorem \ref{th1}.

\subsection{A preliminary result}
The following proposition will be used repeatedly in the proof of Theorem \ref{th1}. It might be well-known in the literature, but nevertheless we provide a detailed proof. 
Below we denote by $\phi_{ \Sigma}$ the density of $\mathcal N_d(0, \Sigma)$. 

\begin{prop} \label{prop1}
We consider random variables $Z \sim \mathcal N_d(0, \Sigma)$ and 
$Z' \sim \mathcal N_d(0, \Sigma')$, where $\Sigma, \Sigma' \in \R^{d\times d}$ are positive definite matrices. We assume that 
there exists a constant $K>0$ with
\[
\max_{1 \leq i,j \leq d} \{|\Sigma_{ij}| + |\Sigma'_{ij}| \} <K \qquad 
\text{and} \qquad \min\{\det{\Sigma}, \det{\Sigma'}\}> \frac{1}{K}.
\]  
Let $G: \R^d \to \R$ be a function with polynomial growth. Then there exist constants $c_K, C_K>0$
such that 
\begin{align*}
\left| \E[G(Z)] -  \E[G(Z')] \right| &\leq C_K \int_{\R^d} G(y) (1+\|y\|^2) \exp(-c_K \|y\|^2) dy 
\\[1.5 ex]
&\times
\max_{1 \leq i,j \leq d} \{|\Sigma_{ij} - \Sigma'_{ij}|\} .
\end{align*}
\end{prop}
\begin{proof}
Let $\lambda_1 \geq \cdots \geq \lambda_d>0$ denote the real eigenvalues of the matrix $\Sigma$. Since 
$\sum_{i=1}^d \lambda_i = \text{tr}(\Sigma)$, we have that $\lambda_i <dK$ for all $i$. 
The same inequality holds
for the eigenvalues $\lambda'_1\geq \cdots \geq \lambda'_d>0$ of $\Sigma'$. 
We also have that
\begin{align} \label{det}
|\det(\Sigma) - \det(\Sigma')| \leq C_K \max_{1 \leq i,j \leq d} \{|\Sigma_{ij} - \Sigma'_{ij}|\},
\end{align}
because the determinant is a polynomial function. In the next step, we give an estimate on 
$\max_{1 \leq i,j \leq d} \{|\Sigma_{ij}^{-1} - \Sigma'^{-1}_{ij}|\}$. Recall the Cramer's rule:
\[
\Sigma_{ij}^{-1} = \frac{\det(\Sigma(i,j))}{\det(\Sigma)},
\]
where the matrix $\Sigma(i,j)$ is formed from $\Sigma$ by replacing the $i$th column of $\Sigma$ by the $j$th standard basis element $e_j\in \R^d$. Applying this formula and using the lower bound 
$\min\{\det{\Sigma}, \det{\Sigma'}\}> 1/K$, we conclude  that 
\begin{align} \label{inverse} 
\max_{1 \leq i,j \leq d} \{|\Sigma_{ij}^{-1} - \Sigma'^{-1}_{ij}|\} \leq C_K \max_{1 \leq i,j \leq d} \{|\Sigma_{ij} - \Sigma'_{ij}|\} 
\end{align}
as in \eqref{det}. We are now ready to obtain an upper bound for $|\phi_{ \Sigma}(y) -
\phi_{ \Sigma'}(y) |$, $y\in \R^d$.  First, we note that 
\begin{align*}
&|\phi_{\Sigma}(y) -
\phi_{ \Sigma'}(y) | \leq C_K \Big(\max_{1 \leq i,j \leq d} \{|\Sigma_{ij} - \Sigma'_{ij}|\}
\Big(\exp(-y^{\star} \Sigma^{-1} y/2)  \\[1.5 ex]
&+ \exp(-y^{\star} \Sigma'^{-1}/2) \Big) + \Big|\exp(-y^{\star} \Sigma^{-1} y/2)  
- \exp(-y^{\star} \Sigma'^{-1} y/2) \Big| \Big),
\end{align*} 
where $^{\star}$ denotes the transpose operator. Now, using mean value theorem and \eqref{inverse} we deduce that 
\begin{align*}
&\Big|\exp(-y^{\star} \Sigma^{-1} y/2) - \exp(-y^{\star} \Sigma'^{-1} y/2) \Big| 
\\[1.5 ex]
&\leq C_K \|y\|^2  
\max_{1 \leq i,j \leq d} \{|\Sigma_{ij} - \Sigma'_{ij}|\} 
 \Big(\exp(-y^{\star} \Sigma^{-1} y/2) + \exp(-y^{\star} \Sigma'^{-1} y/2) \Big). 
\end{align*}
Next, we note that 
\begin{align} \label{qform}
y^{\star} \Sigma^{-1} y = \|\Sigma^{-1/2} y\|^2 \geq \|y\|^2/ \lambda_1 \geq \|y\|^2/dK 
\end{align}
and the same inequality holds for the matrix $\Sigma'^{-1}$. 
Since
\[
\left| \E[G(Z)] -  \E[G(Z')] \right|  \leq \int_{\R^d} G(y) |\phi_{\Sigma}(y) -
\phi_{ \Sigma'}(y) | dy,
\]
we obtain the desired result. 
\end{proof}

\subsection{Proof of \eqref{convergence1}} 
We define the random variables
\[
r_i^n:=f \left(u_n X_{i/n}, n^H \Delta_i^n X \right) - F(u_n X_{i/n}) .
\]
We will now show that $(u_n/n)^2 \E[(\sum_{i=0}^n r_i^n)^2] \to 0$. We divide this proof into two parts. We will prove that
\begin{align} \label{conv1}
& \frac{u_n^2}{n^2} \sum_{i=0}^n \E[(r_i^n)^2] \to 0, \\[1.5 ex]
\label{conv2} &  \frac{u_n^2}{n^2} \sum_{j>i}^n \E[r_i^n r_j^n] \to 0. 
\end{align} 
We start with the first statement. Using the covariance kernel of the fractional Brownian motion we see that 
\begin{align} \label{corr}
\rho_i:= \text{corr}(X_{i/n}, \Delta_i^n X) = \frac{1}{2i^{H}} \left( (i+1)^{2H} - i^{2H} -1 \right)
\qquad \text{for } i \geq 1. 
\end{align}
We see immediately that $|\rho_i| \to 0$ as $i \to \infty$ and $|\rho_i| < \rho$ for all $i \geq 1$ and some $\rho<1$.  Let us now consider the correlation matrix 
\[
\Sigma_{i} =
\left(
\begin{array} {cc}
1 & \rho_i \\
\rho_i & 1
\end{array}
\right).
\]
This matrix has two eigenvalues $\lambda_{i}(1) = 1+|\rho_i|$ and $\lambda_{i}(2) = 1-|\rho_i|$. Hence, the matrix $\Sigma_{i}^{-1/2}$ has eigenvalues $\lambda_{i}(1)^{-1/2}$ and $\lambda_{i}(2)^{-1/2}$,
and we conclude that 
\begin{align} \label{lowerbound}
(x,z)^{\star} \Sigma_i^{-1} (x,z) = \|  \Sigma_i^{-1/2} (x,z) \|^2 \geq  (x^2 +z^2)/ \lambda_{i}(1)
\geq  (x^2 +z^2)/(1+\rho).
\end{align}   
For $i=0$ we obviously have 
\begin{align*}
\E\left[f^2 \left(0, n^H \Delta_0^n X \right) \right] \leq C.
\end{align*}
On the other hand, for 
$1\leq i \leq  n$ we obtain the estimate
\begin{align*}
\E[f^2(u_n X_{i/n}, n^H \Delta_i^n X)] &= \int_{\R^2} f^2(u_n (i/n)^H x, z) \phi_{0, \Sigma_i} (x,z) dxdz
\\[1.5 ex]
&=  u_n^{-1} (i/n)^{-H} \int_{\R^2} f^2(y, z) \phi_{0, \Sigma_i} (u_n^{-1} (i/n)^{-H} y,z) dydz
\\[1.5 ex]
&\leq C u_n^{-1} (i/n)^{-H}  \int_{\R^2} f^2(y, z) \exp(-z^2/2) dydz
\\[1.5 ex]
&\leq C u_n^{-1} (i/n)^{-H},  
\end{align*}
where we have used \eqref{lowerbound} and the condition \eqref{assumption} applied to $h_1^2$ 
in the two last steps. Since the mapping $x \mapsto x^{-H}$ is integrable around zero, we deduce that 
\[
\frac{u_n^2}{n^2} \sum_{i=0}^n \E[f^2(u_n X_{i/n}, n^H \Delta_i^n X)] \leq C \frac{u_n}{n} \to 0.
\]
Similarly we obtain the convergence 
\[
\frac{u_n^2}{n^2} \sum_{i=0}^n \E[F^2(u_n X_{i/n})]  \to 0,
\]	
which concludes the proof of \eqref{conv1}.

In order to prove \eqref{conv2}, we need several decompositions. Let us fix a $\delta \in (0,1)$ and write
\[
 \frac{u_n^2}{n^2} \sum_{j>i}^n \E[r_i^n r_j^n] = \sum_{k=1}^3 R_{n,\delta}(k),
\]
where
\begin{align*}
R_{n,\delta}(1)&:= \frac{u_n^2}{n^2} \sum_{i=0}^{[n\delta]} \sum_{j=i+1}^n \E[r_i^n r_j^n], \\[1.5 ex]
R_{n,\delta}(2)&:= \frac{u_n^2}{n^2} \sum_{i=[n\delta]}^{n} \sum_{j-i \leq [n\delta]}^n \E[r_i^n r_j^n] , 
\\[1.5 ex]
R_{n,\delta}(3)&:= \frac{u_n^2}{n^2} \sum_{i=[n\delta]}^{n} \sum_{j-i > [n\delta]}^n \E[r_i^n r_j^n]. 
\end{align*}
We note that 
\[
\left(B_{i,n}, B_{j/n}, n^{H} \Delta_i^n X, n^{H} \Delta_j^n X  \right) \sim \mathcal{N}_4 (0, \Sigma_{i,j}^n)
\]
and we denote by $\widetilde{\Sigma}_{i,j}$ the correlation matrix associated with $\Sigma_{i,j}^n$, which does not depend on $n$. We obviously have that $|\Sigma_{i,j}^n(k,l)| \leq C$ and 
$|\widetilde{\Sigma}_{i,j}(k,l)| \leq 1$ for any $i,j$ and any $1 \leq k,l \leq 4$. For any $s \in (0,1)$ 
and $1\leq i <j \leq n$ we have the identities 
\begin{align*}
&\text{cov} (X_s, n^H \Delta_i^n X) = \frac{n^H}{2} \left(\left( \frac{i+1}{n}\right)^{2H} - 
\left( \frac{i}{n}\right)^{2H} + \left| \frac{i}{n} - s \right|^{2H} -  \left| \frac{i+1}{n} - s \right|^{2H} \right),
\\[1.5 ex]
&\text{cov} ( n^H \Delta_i^n X, n^H \Delta_j^n X)=\frac{1}{2} \left( (j-i-1)^{2H} + (j-i+1)^{2H} 
-2 (j-i)^{2H} \right).
\end{align*}
Using these identities  we observe the following properties:
\begin{align} \label{property1}
&\inf_{1\leq i<j \leq n}\det(\widetilde{\Sigma}_{i,j})>c>0,  \\[1.5 ex]
\label{property2}
&\det(\Sigma_{i,j}^n)>c_{\delta}>0 \qquad \text{for any }  i>[n\delta] \text{ and } j-i>[n\delta], \\[1.5 ex]
\label{property3}
& |\Sigma_{i,j}^n(k,l)|  \leq a_{n, \delta} 
\qquad \text{for any }  i>[n\delta] \text{ and } j-i>[n\delta],
\end{align}
where $a_{n,\delta} \to 0$ as $n \to \infty$ and $(k,l) \in J$ with
\[
J=\{1,\ldots, 4\}^2 \setminus \left( \cup_{k=1}^4 \{k,k\} \cup \{1,2\} \cup \{2,1\} \right)
\]
Now, we start with the term $R_{n,\delta}(1)$. For $1\leq i \leq [n \delta]$ and $i<j\leq n$, 
we use \eqref{property1} and \eqref{qform},  and
deduce that ($x=(x_1,x_2)$ and $z=(z_1,z_2)$)
\begin{align*}
&\E[|f(u_n X_{i/n}, n^H \Delta_i^n X) f(u_n X_{j/n}, n^H \Delta_j^n X)| ] \\[1.5 ex]
&= \int_{\R^4} 
|f(u_n (i/n)^H x_1, z_1) f(u_n (j/n)^H x_2, z_2)| \phi_{\widetilde{\Sigma}_{i,j}} (x,z) dxdz
\\[1.5 ex]
& \leq C
\int_{\R^4} 
|f(u_n (i/n)^H x_1, z_1) f(u_n (j/n)^H x_2, z_2)| \exp(-c(\|x\|^2 + \|z\|^2)) dxdz
\\[1.5 ex]
&\leq C u_n^{-2} (i/n)^{-H}  (j/n)^{-H}  \int_{\R^4} |f(y_1, z_1) f(y_2, z_2)| \exp(-c\|z\|^2) dydz
\\[1.5 ex]
&\leq C u_n^{-2} (i/n)^{-H}  (j/n)^{-H}.    
\end{align*}
The last inequality follows from condition (A-$2$):
\[
 \int_{\R^4} |f(y_1, z_1) f(y_2, z_2)| \exp(-\|z\|^2/2) dydz \leq \|h_1\|_{L^1(\R)}^2 
 \int_{\R^2} |h_2( z_1)h_2( z_2)| \exp(-c\|z\|^2) dz.
\]
The other terms in $\E[r_i^n r_j^n]$ for  $1\leq i \leq [n \delta]$ and $i<j\leq n$ are treated in a similar way. We thus conclude that 
\begin{align} \label{r1est}
|R_{n,\delta}(1)| \leq C \delta^{1-H}.
\end{align}
By the same arguments we also deduce that 
\begin{align} \label{r2est}
|R_{n,\delta}(2)| \leq C \delta^{1-H}.
\end{align}
Next we compute the term $R_{n,\delta}(3)$. We define the matrix $\Sigma_{i,j}^{ n,1}=
(\Sigma_{i,j}^{n,1}(k,l))_{1 \leq k,l \leq 4}$ via
\[
\Sigma_{i,j}^{n,1}(k,l) = \Sigma_{i,j}^n(k,l) 1_{J^c}.
\]
By properties \eqref{property2}-\eqref{property3} and 
Proposition \ref{prop1} we deduce for any   $i>[n\delta]$ and $j-i>[n\delta]$
\begin{align*}
&\E[f(u_n X_{i/n}, n^H \Delta_i^n X) f(u_n X_{j/n}, n^H \Delta_j^n X) ] \\[1.5 ex]
&= \int_{\R^4} 
f(u_n x_1, z_1) f(u_n  x_2, z_2) \phi_{\Sigma_{i,j}^{n,1}} (x,z) dxdz
+ w_{i,j}^{n,1} = : \bar{w}_{i,j}^{n,1} + w_{i,j}^{n,1} ,    
\end{align*}
where
\begin{align*}
|w_{i,j}^{n,1}| &\leq C_{\delta} a_{n,\delta}  
\int_{\R^4} 
f(u_n  x_1, z_1) f(u_n  x_2, z_2) (1+\|(x,z)\|^2) \exp(-c_{\delta} \|(x,z)\|^2) dxdz. \\[1.5 ex]
& \leq 
C_{\delta} a_{n,\delta}  u_n^{-2} \left( \int_{\R} (1+x_1^2) h_1(x_1) dx_1 \right)^2 
\left(\int_{\R} h_2(z_1) z_1^2 \exp(-c_{\delta} z_1^2) dz_1 \right)^2. 
\end{align*}
Now we define the matrix $\Sigma_{i,j}^{ n,2} \in \R^{3 \times 3}$  (resp. $\Sigma_{i,j}^{ n,3} \in \R^{3 \times 3}$ and $\Sigma_{i,j}^{ n,4}  \in \R^{2 \times 2}$) through the matrix $\Sigma_{i,j}^{ n,1}$ by deleting the fourth row/column (resp. the third row/column and the last two rows/columns). We similarly obtain the identities
\begin{align*}
&\E[f(u_n X_{i/n}, n^H \Delta_i^n X) F(u_n X_{j/n}) ] \\[1.5 ex]
&= \int_{\R^3} 
f(u_n x_1, z_1) F(u_n  x_2) \phi_{\Sigma_{i,j}^{n,2}} (x,z_1) dxdz_1
+ w_{i,j}^{n,2} = : \bar{w}_{i,j}^{n,2} + w_{i,j}^{n,2}, \\[1.5 ex]
&\E[f(u_n X_{j/n}, n^H \Delta_j^n X) F(u_n X_{i/n}) ] \\[1.5 ex]
&= \int_{\R^3} 
F(u_n x_1) f(u_n  x_2, z_2) \phi_{\Sigma_{i,j}^{n,3}} (x,z_2) dxdz_2
+ w_{i,j}^{n,3} = : \bar{w}_{i,j}^{n,3} + w_{i,j}^{n,3}, \\[1.5 ex]
&\E[F(u_n X_{i/n}) F(u_n X_{j/n}) ] \\[1.5 ex]
&\int_{\R^2} 
F(u_n x_1) F(u_n  x_2) \phi_{\Sigma_{i,j}^{n,4}} (x) dx
+ w_{i,j}^{n,4} = : \bar{w}_{i,j}^{n,4} + w_{i,j}^{n,4}
\end{align*}   
and 
\begin{align*}
|w_{i,j}^{n,2}| + |w_{i,j}^{n,3}| + |w_{i,j}^{n,4}| \leq C_{\delta} a_{n,\delta}.  
\end{align*}
In fact, we obviously have 
\[
\bar{w}_{i,j}^{n,1}=\bar{w}_{i,j}^{n,2}=\bar{w}_{i,j}^{n,3}=\bar{w}_{i,j}^{n,4}.
\] 
Hence, we conclude that 
\begin{align} \label{r3est}
|R_{n,\delta}(3)| \leq C_{\delta} a_{n,\delta}.
\end{align}
Thus, the statement of Theorem \ref{th1} follows from \eqref{r1est}-\eqref{r3est} by letting $n \to \infty$
and then $\delta \to 0$.

\subsection{Proof of Proposition \ref{prop2}}

Observe the identity 
$$(\Delta_i^n |X|)^2  - (\Delta_i^n X)^2= 2|X_{i/n} X_{(i+1)/n}| 1_{\{X_{i/n} X_{(i+1)/n}<0\}}.$$
Thus, we have 
\begin{align*}
\frac{1}{n} \sum_{i=0}^{[nt]} \left(  (n^H \Delta_i^n |X|)^2 -\sigma^2 \right)
&= \frac{1}{n} \sum_{i=0}^{[nt]} \left(  (n^H \Delta_i^n X)^2 -\sigma^2 \right) + n^{-H}
V(f_2)_t^n. \\[1.5 ex]
& =: M_t^n + n^{-H} V(f_2)_t^n
\end{align*}
It is well known that for $H \in (0,3/4)$ we have the functional convergence
\[
\sqrt{n} M_t^n \Rightarrow 
vW_t,
\]
while for $H\in (3/4,1)$ it holds that 
\[
n^{2-2H} M_t^n \Rightarrow \sigma^2 R_t,
\] see \cite{BM,DM,T}.
Hence, we deduce the assertion of Proposition \ref{prop2}.

\subsection*{Acknowledgements} 
Mark Podolskij acknowledges financial support from the project 
``Ambit fields: probabilistic properties and statistical inference'' funded by Villum Fonden 
and from CREATES funded by the Danish
National Research Foundation. Mathieu Rosenbaum acknowledges financial support of the ERC
679836 Staqamof.

\end{document}